\documentclass[11pt]{article}
\usepackage{amsmath}
\usepackage{amssymb,bbm}
\usepackage{graphicx}
\usepackage{latexsym,url}
\usepackage{authblk}

\newcommand{\oR}{{\mathbb R}}
\newcommand{\oN}{{\mathbb N}}
\newcommand{\oQ}{{\mathbb Q}}
\newcommand{\oZ}{{\mathbb Z}}

\newcommand{\EE}{\mathcal E}
\newcommand{\SSS}{{\mathcal S}}

\newcommand{\gd}{\text{\rm gd}}

\newcommand{\CUT}{\text{\rm CUT}}
\newcommand{\MET}{\text{\rm MET}}

\newcommand{\rank}{{\text{\rm rank}}}

\newcommand{\conv}{\text{\rm conv}}

\newcommand{\al}{\alpha}

\newcommand{\ext}{\text{\rm ext}}

\newcommand{\bc}{\begin{center}}
\newcommand{\ec}{\end{center}}
\newcommand{\sfT}{{\sf T}}
\newcommand{\NP}{{\mathcal NP}}
\newcommand{\bfO}{\mathbf 0}
\newcommand{\convEk}{\conv\hspace*{0.05cm}   \EE_k(G)}
\newcommand{\spp}{\hspace*{0.05cm}}
\newcommand{\haG}{\widehat{G}}
\newcommand{\haE}{\widehat{E}}
\newcommand{\hax}{\widehat{x}}
\newcommand{\haV}{\widehat{V}}

\newcommand{\wt}{\widetilde}
\newcommand{\ed}{\text{\rm ed}}
\newcommand{\tG}{\tilde G}

\newtheorem{theorem}{Theorem}[section]
\newtheorem{claim}{Claim}[section]

\newtheorem{definition}{Definition}[section]
\newtheorem{corollary}{Corollary}[section]

\newtheorem{proposition}{Proposition}[section]
\newtheorem{lemma}{Lemma}[section]
\newtheorem{remark}{Remark}[section]

\newcommand{\qed}{\hspace*{\fill} $\Box$  \medskip}
\newcommand{\ignore}[1]{}

\title{
Complexity of  the positive semidefinite matrix completion problem with a rank constraint}

\author[1]{ M. E.-Nagy\thanks{M.E.Nagy@cwi.nl}}
\author[1,2]{M. Laurent\thanks{M.Laurent@cwi.nl}}
 \author[1]{A. Varvitsiotis\thanks{A.Varvitsiotis@cwi.nl}}
\affil[1]{Centrum Wiskunde \& Informatica (CWI), Amsterdam, The Netherlands.}
\affil[2]{Tilburg University, Tilburg, The Netherlands.}

\begin{document}

\maketitle

\begin{abstract}
We consider the decision problem asking whether a partial rational symmetric matrix with an all-ones diagonal 
can be completed to a full  positive semidefinite matrix  of rank at most $k$.
  We show that  this problem is $\NP$-hard for any fixed integer $k\ge 2$. Equivalently, for $k\ge 2$, it is $\NP$-hard to test membership in the rank constrained elliptope $\EE_k(G)$, i.e.,  the set  of all partial matrices with off-diagonal entries specified at the edges of $G$, that can be completed to a positive semidefinite matrix of rank at most $k$.
Additionally,  we show  that deciding  membership in the convex hull of $\EE_k(G)$ is also $\NP$-hard for any fixed integer  $k\ge 2$.

\end{abstract}

\section{Introduction}

Geometric representations of graphs are widely studied within a broad range of mathematical areas, ranging from combinatorial matrix theory, linear algebra, discrete geometry, and combinatorial optimization. 
They arise typically  when labeling the nodes by vectors assumed to satisfy certain properties. For instance, one may require that the vectors labeling adjacent nodes are at distance 1, leading to unit distance graphs.
Or one may require that the vectors labeling adjacent nodes are orthogonal, leading to orthogonal representations of graphs.
One may furthermore ask, e.g., that nonadjacent nodes receive vector labels that are not orthogonal. Many other geometric properties of orthogonal labelings and other types of  representations related, e.g., to Colin de Verdi\`ere type graph parameters,
 are of interest and have been investigated (see \cite{CdV}). A basic question is to determine the smallest possible dimension of such  vector representations.
There is a vast literature, we refer in particular to the surveys \cite{FH07,FH11,Lo01}  and further references therein for additional information. 

\ignore{

 {\em unit distance graphs} (introduced by Erd\"os, Harary and Tutte \cite{EHT}) arise when  the vectors labeling adjacent nodes are assumed to be at distance 1. Various questions arise: What is the chromatic number? In the 2D case,  the chromatic number  is known to belong to $\{4,5,6,7\}$ (upper bound 7 proved by Hadwiger \cite{Ha}) but the exact value remains open. 
What is the smallest dimension $d(G)$ needed to realize the graph? It is  linked to the chromatic number of $G$, since
$\log \chi(G)\le d(G)\le 2\chi(G),$ but determing $d(G)$ 
 is an $\mathcal NP$-hard problem ({\bf add ref.}).
Or what is the smallest radius of a ball containing a unit distance representation? This parameter is easier as it can be computed with a semidefinite program, it is intimately linked to $\vartheta(G)$, the theta number of $G$.

Numerous other geometric properties of vector labelings are of interest. We mention in particular the bulk of literature on graph parameters related -- and motivated by -- the graph parameter $\mu(G)$ introduced by Colin de Verdi\`eres \cite{CdV}, see e.g. the survey \cite{} for details and further references.   {\em Orthogonal representations},  where the vectors labeling adjacent nodes are required to be orthogonal,  arise naturally in many contexts -- for instance in connection to the Lov\'asz' theta number $\vartheta(G)$,  bounding the stability number $\alpha(G)$   and the chromatic number $\chi(G)$.  In fact they are special instances of Gram embeddings of graphs, introduced in \cite{LV11} and further discussed in this note.
}

In this note we revisit orthogonal representations of graphs, in the wider context of Gram representations of weighted graphs. 
We show some complexity results for the following notion of Gram dimension,  which has been considered 
in \cite{LV11,LV12}. 

\begin{definition}\label{defGram1}
Given a graph $G=(V=[n],E)$ and  $x\in \oR^E$, a Gram representation of $x$ in $\oR^k$ consists of a set of unit vectors $v_1,\cdots, v_n\in \oR^k$ such that 
$$v_i^\sfT v_j = x_{ij} \ \ \forall \{i,j\}\in E.$$
The {\em Gram dimension} of $x$, denoted as $\gd(G,x)$, is the smallest integer $k$ for which $x$ has  such a Gram representation in $\oR^k$ (assuming it has one in some space). 
\end{definition}

As we restrict our attention to Gram representations of $x\in \oR^E$ by unit vectors, all coordinates of $x$ should lie in the interval $[-1,1]$, so that we can parametrize $x$ as 
$$x=\cos (\pi a), \ \ \ \text{ where } a\in [0,1]^E.$$
In other words,  the inequality $\gd(G,x)\le k$ means that $(G,a)$ can be isometrically embedded into the spherical metric space $({\mathbf S}^{k-1},d_{\mathbf S})$, where ${\mathbf S}^{k-1}$ is the unit sphere in the Euclidean space $\oR^k$ and $d_{\mathbf S}$ 
is the spherical distance:
$$d_{\mathbf S}(u,v)=\arccos(u^\sfT v)/\pi\ \ \forall u,v\in {\mathbf S}^{k-1}.$$
Moreover, there are  also tight connections with  graph realizations in the Euclidean space (cf. 
 \cite{B07,BC07}); see Section \ref{secedm} for a brief discussion and Section \ref{seck2} for further results.

\medskip
 Determining the Gram dimension  can also be reformulated in terms of finding low rank positive semidefinite matrix completions of partial matrices, as we now see.
We use  the following notation:
 $\SSS^n$ denotes the set of symmetric $n\times n$ matrices and $\SSS^n_+$ is the cone of positive semidefinite (psd) matrices in $\SSS^n$. The subset 
 $$\EE_n=\{X\in \SSS^n_+: X_{ii}=1 \ \forall i\in [n]\},$$
 consisting of all positive semidefinite  matrices with an all-ones diagonal (aka the correlations matrices), is known as the {\em elliptope}.
 Given a graph $G=([n],E)$, $\pi_E$ denotes the projection from $\SSS^n$ onto the subspace $\oR^E$ indexed by the edges of $G$. Then, the projection $\EE(G)=\pi_E(\EE_n)$ is known as the {\em elliptope} of the graph $G$. Given an integer $k\ge 1$,
 define the {\em rank constrained elliptope}
 $$\EE_{n,k}=\{X\in \EE_n: \rank (X)\le k\},$$
 and, for any graph $G$,  its projection 
 $\EE_k(G)=\pi_E(\EE_{n,k})$.
 Then the points  $x$ in the elliptope $\EE(G)$ correspond  precisely to those vectors $x\in\oR^E$ that admit a Gram representation by unit vectors. Moreover, $x\in \EE_k(G)$ precisely when it has  a Gram representation by unit vectors in $\oR^k$; that is:
 $$x\in \EE_k(G)\Longleftrightarrow \gd(G,x)\le k.$$

 The elements of $\EE(G)$ can be seen as the $G$-partial symmetric matrices, i.e., the partial matrices whose entries are specified at the off-diagonal positions corresponding to edges of $G$ and whose diagonal entries are all equal to 1, that can be completed to a positive semidefinite matrix.
 Hence the problem of deciding membership in $\EE(G)$ can be reformulated as the problem of testing whether a given $G$-partial matrix can be completed to a psd matrix.
 Moreover, for fixed $k\ge 1$, the membership problem in $\EE_k(G)$ is  the problem of deciding whether a given $G$-partial matrix has a psd completion of rank at most $k$.  Using the notion of Gram dimension this can be equivalently formalized as:  
 
\medskip\noindent
 {\em  Given a graph $G=(V,E)$ and  $x\in \oQ^E$, decide whether $\gd(G,x)\le k$.}
 
 \medskip
 A first main result of our paper is to prove $\NP$-hardness of this  problem for any fixed $k\ge 2$ (cf. Theorems \ref{theomaink3} and \ref{theogd2}).
Additionally,  we  consider the problem of testing membership in the convex hull of the rank constrained elliptope:  
 
 \medskip\noindent
 {\em  Given a graph $G=(V,E)$ and $x\in \oQ^E$, decide whether $x\in \convEk$.}
 
\medskip
The study of this problem  is motivated by the relevance of the convex set  $\convEk$  to the maximum cut problem and to the rank constrained Grothendieck problem.
Indeed, for $k=1$, $\conv \spp \EE_1(G)$ coincides with the cut polytope of $G$ and it is well known that linear optimization over the cut polytope is  $\NP$-hard ~\cite{GJ79}.
For  any $k\ge 2$,   the worst case ratio of optimizing a  linear  function over the elliptope $\EE(G)$ versus   the rank constrained elliptope $\EE_k(G)$ (equivalently, versus the convex hull $\convEk$)  is known as the {\em rank $k$ Grothendieck constant}  of the graph $G$ (see \cite{BOV11} for  results and further references).   It is believed that linear optimization over $\convEk$ is also hard for any fixed $k$ (cf., e.g., the quote of Lov\'asz \cite[p. 61]{Lo01}).We show that the strong membership problem in $\convEk$ is $\NP$-hard, thus providing some evidence of hardness of optimization (cf. Theorem \ref{theoextgd}).

\medskip
{\bf Contents of the paper.}
In Section \ref{secpreli} we present some background geometrical  facts about cut and metric polytopes, about elliptopes, and about Euclidean graph realizations.
In  Section \ref{secEk} we show $\NP$-hardness of the membership problem in $\EE_k(G)$ for any fixed $k\ge 2$;  we use two different reductions  depending whether $k=2$ or $k\ge 3$. 
In Section \ref{secconvEk} we show $\NP$-hardness of the membership problem in
 the  convex hull of $\EE_k(G)$  for any fixed $k\ge 2$.
In Section \ref{secedm} we discuss links to complexity results for Euclidean graph realizations, and in Section \ref{secfinal} we conclude with some open questions.

\medskip
{\bf Some notation.} Throughout $K_n=([n],E_n)$ is the complete graph on $n$ nodes;
$C_n$ denotes the circuit of length $n$, with node set $[n]$ and with edges the pairs $\{i,i+1\}$ for $i\in [n]$ (indices taken modulo $n$), and its set of edges is again denoted as  $C_n$ for simplicity.
Given a graph $G=(V,E)$, its {\em suspension graph} $\nabla G$ is the new graph obtained from $G$ by adding a new node, called the {\em apex} node and often denoted as $0$, which is adjacent to all the nodes of $G$.
A {\em minor} of $G$ is any graph which can be obtained from $G$ by iteratively deleting edges or nodes and contracting edges.

\section{Preliminaries}\label{secpreli}
We recall here some  basic geometric  facts about metric and cut polyhedra, about elliptopes, and  about Euclidean graph realizations.

\subsection{Metric and cut polytopes}

First we recall the definition of the {\em metric polytope} $\MET(G)$ of a graph $G=(V,E)$.
As a motivation recall the following basic 3D geometric result:  Given a  matrix   $X=(x_{ij})$ in the elliptope  $ \EE_3$, 
 parametrized as before by   $x_{ij}=\cos (\pi a_{ij})$ where  $a_{ij}\in [0,1]$, then  $X\succeq 0$ if and only if  the $a_{ij}$'s  satisfy the following {\em triangle inequalities}:
\begin{equation}\label{eqtriangle}
a_{ij}\le a_{ik}+a_{jk},\ a_{ij}+a_{ik}+a_{jk}\le 2\ \
\end{equation}
for distinct $i,j,k\in \{1,2,3\}$.
(See e.g. \cite{BJL96}).  The elliptope $\EE_3$ (or rather, its bijective image $\EE(K_3)$) 
 is illustrated in Figure~\ref{E3}.

The metric polytope of the complete graph $K_n=([n],E_n)$  is the polyhedron in $\oR^{E_n}$ defined by the above $4{n\choose 3}$ triangle inequalities
(\ref{eqtriangle}). 
More generally, 
the metric polytope of a graph $G=([n],E)$ is the polyhedron $\MET(G)$ in  $\oR^E$, which is defined by the following linear inequalities (in the variable $a\in\oR^E$):
\begin{equation}\label{reledge}
0\le a_e\le 1\ \ \forall e\in E,
\end{equation}
\begin{equation}\label{relcyc}
a(F)-a(C\setminus F) \le |F|-1 
\end{equation}
for all circuits  $C$ of $G$ and for all odd cardinality subsets $F\subseteq C$.

As is well known, the inequality (\ref{reledge}) defines a facet of $\MET(G)$ if and only if the edge $e$ does not belong to a triangle of $G$, while (\ref{relcyc}) defines a facet of $\MET(G)$ if and only if the circuit $C$ has no chord (i.e., two non-consecutive nodes on $C$ are not adjacent in $G$).
In particular, for  $G=K_n$, $\MET(K_n)$ is defined by the  triangle inequalities (\ref{eqtriangle}), obtained by considering only the inequalities (\ref{relcyc}) where  $C$ is a circuit of length 3.
Moreover, $\MET(G)$  coincides with the projection of $\MET(K_n)$ onto the subspace $\oR^E$ indexed by the edge set of $G$. (See \cite{DL97} for details.)

\medskip
A main motivation for considering the metric polytope is that it gives a tractable linear relaxation of the cut polytope.
Recall that the rank 1 matrices in the elliptope $\EE_n$ are of the form $uu^\sfT$ for all $u\in \{\pm 1\}^n$, they are sometimes called the {\em cut matrices} since they   correspond to the cuts of the complete graph $K_n$. Then the {\em cut polytope} $\CUT(G)$ is defined as the projection onto $\oR^E$ of the convex hull of the cut matrices: 
\begin{equation}\label{eqcut}
\CUT(G)=\pi_E(\conv(\EE_{n,1})).
\end{equation}
It is always true that  $\CUT(G)\subseteq \MET(G)$, moreover  equality holds if and only if $G$ has no $K_5$ minor  \cite{BM}.
Linear optimization over the cut polytope models the maximum cut problem, 
 well known to be $\NP$-hard  \cite{GJ79},  and  testing membership in the cut polytope $\CUT(K_n)$ or, equivalently,  in the convex hull of the rank constrained elliptope $\EE_{n,1}$, is an $\NP$-complete problem  \cite{AD91}.

\subsection{Elliptopes}

From the above discussion about the elliptope $\EE_3$ and the metric polytope, we can derive the following necessary condition for membership in the elliptope $\EE(G)$ of a graph $G$, which turns out to be sufficient  when $G$ has no  $K_4$ minor. 

 \begin{proposition}\label{propEG}
 \cite{Lau97} For any graph $G=(V,E)$,
$$\EE(G)\subseteq \left\{x\in [-1,1]^E: {1\over \pi}\arccos x \in \MET(G)\right\}.$$
Moreover, equality holds if and only if   $G$ has no $K_4$ minor.
\end{proposition}

\ignore{
This result was used in \cite{Lau00} to give polynomial time algorithms (in the real number model of computation) to construct  positive semidefinite matrix completions for  graphs with no $K_4$ minor.}
This result permits, in particular, to characterize membership in the elliptope  $\EE(C_n)$ of a circuit.

 \begin{corollary}\cite{BJL96}
Consider a vector $x=\cos(\pi a) \in \oR^{C_n}$ with $a\in [0,1]^{C_n}$. Then, $x\in \EE(C_n)$  if and only if  $a$ satisfies the linear inequalities
\begin{equation}\label{eqcircuit}
a(F)-a(C_n\setminus F)\le |F|-1\ \ \forall F\subseteq C_n \ \text{ with  } |F|  \text{ odd}.
\end{equation}
\end{corollary}

We also recall the following result of \cite{LV11} which characterizes membership in the rank constrained elliptope $\EE_k(C_n)$ of a circuit in the case $k=2$; see Lemma \ref{lemgd2} for an extension to arbitrary graphs.

\begin{lemma}\cite{LV11}\label{lemE2Cn}
For $x\in [-1,1]^{C_n}$, $x\in \EE_2(C_n)$ if and only if  there exists $\epsilon \in \{\pm 1\}^{C_n}$ such that 
$\epsilon ^\sfT \arccos x \in 2\pi \oZ$.
\end{lemma}

\ignore{
\begin{proof}
We first prove necessity. Assume that $x\in \EE_2(C_n)$. That is, there exist unit vectors $u_1,\ldots,u_n\in \oR^2$ such that
$u_i^\sfT u_{i+1}= \cos (\pi a_i)$ for all $i\in [n]$ (setting $u_{n+1}=u_1$).
Up to applying an orthogonal transformation, we may assume that  $u_1=(1,  0)^\sfT$.
Then, $u_1^\sfT u_2=\cos (\pi a_1)$,   and thus  $$u_2=(\cos (\epsilon_1 \pi a_1),
\sin(\epsilon_1 \pi a_1))^\sfT$$ for some $\epsilon_1\in\{\pm 1\}$.
Analogously, $u_2^\sfT u_3=\cos (\pi a_2)$, which  implies that
 $$u_3=(\cos(\epsilon_1\pi a_1+\epsilon_2\pi a_2 ), \sin(\epsilon_1\pi a_1+\epsilon_2\pi a_2))^\sfT$$ for some $\epsilon_2\in \{\pm 1\}$.
Iterating, we find  $\epsilon_1,\cdots, \epsilon_n \in\{\pm 1\}$ such that
$$u_{i}=\left(\cos\left(\sum_{j=1}^{i-1}\epsilon_i \pi a_i\right),  \sin \left(\sum_{j=1}^{i-1}\epsilon_i \pi a_i\right)\right)^\sfT$$ for $i=1,\ldots,n$.
Finally, the condition $u_n^\sfT u_1=\cos (\pi a_n)= \cos (\sum_{i=1}^{n-1} \epsilon_i \pi a_i)$
 implies that $\sum_{i=1}^n\epsilon_i\pi a_i\in 2\pi\oZ$ and thus $\sum_{i=1}^n \epsilon_ia_i\in 2\oZ$.
 The arguments can be reversed to show sufficiency.
 \qed\end{proof}
 }

We conclude with some observations about the elliptope $\EE(C_n)$  of a circuit.
Figure \ref{E3} shows the elliptope $\EE(C_3)$. Points $x$ on the boundary of $\EE(C_3)$ have $\gd(C_3,x)=2$ except $\gd(C_3,x)=1$ at the four corners (corresponding to the four cuts of $K_3$), while points in the interior of $\EE(C_3)$ have $\gd(C_3,x)=3$.

\begin{figure}[h!]
\centering \includegraphics[scale=0.8]{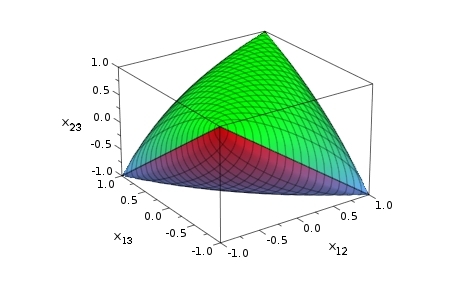}
\caption{The elliptope $\EE(C_3)$.}
\label{E3} 
\end{figure}
Now let $n\ge 4$. Let $x=\cos (\pi a)\in \EE(C_n)$, thus $a\in [0,1]^{C_n}$ satisfies the inequalities (\ref{eqcircuit}).
It is known that $\gd(C_n,x)\le 3$ (see \cite{LV11}, or derive it directly by triangulating $C_n$ and applying Lemma \ref{lemcliquesum} below).
Moreover, $x$ lies in the interior of $\EE(C_n)$ if and only if $x$ has a positive definite completion or, equivalently, $a$ lies in the interior of the metric polytope $\MET(C_n)$. 

If $x$ lies on the boundary of $\EE(C_n)$ then, either (i) $a_e\in \{0,1\}$ for some edge $e$ of $C_n$, or (ii) 
$a$ satisfies an inequality (\ref{eqcircuit}) at equality. In case (i), $\gd(C_n,x)$ can be equal to 1 ($x$ is a cut), 2, or 3. In case (ii), by Lemma \ref{lemE2Cn},
$\gd(C_n,x)\le 2$ since $a(F)-a(C_n\setminus F) =|F|-1 \in 2\oZ$ for some $F\subseteq C_n$. 
If $x$ is in the interior of $\EE(C_n)$ then $\gd(C_n,x)\in \{2,3\}$.

As an illustration, for $n=4$, consider the  vectors $x_1=(0,0,0,1)^\sfT$, $x_2=(0, \sqrt 3/2, \sqrt 3/2, \sqrt 3/2)^\sfT$, 
$x_3=(0,0,0,0)^\sfT$ and $x_4=(0,0,0,1/2)^\sfT \in \oR^{C_4}$.
Then both $x_1$ and $x_2$ lie on the boundary of $\EE(C_4)$ with $\gd(C_4,x_1)=3$ and $\gd(C_4,x_2)=2$, and both $x_3$ and $x_4$ lie in the interior of $\EE(C_4)$ with $\gd(C_4,x_3)=2$ and $\gd(C_4,x_4)=3$.

\subsection{Euclidean graph realizations}\label{secedm}

In this section we recall some basic facts about Euclidean graph realizations.

\begin{definition}\label{defedm}
Given a graph $G=([n],E)$ and  $d\in \oR_{+}^{E}$,  a Euclidean (distance) representation of $d$ in $\oR^k$
consists of  a set of vectors $p_1,\ldots,p_n\in \oR^k$  such that $$\| p_i-p_j\|^2=d_{ij}\ \ \forall  \{i,j\} \in E.$$
Then,   $\ed(G,d)$ denotes the smallest  integer $k\ge 1 $ for which $d$ has a  Euclidean representation in $\oR^k$
(assuming $d$ has a Euclidean representation in some space).
\end{definition}
Then the problem of interest is to decide  whether a given vector  $d \in \oQ_{+}^{E}$ admits a Euclidean representation in $\oR^k$. Formally, for fixed $k \ge 1$, we consider the following problem:

\medskip\noindent
{\em Given a graph $G=(V,E)$ and $d \in \oQ_{+}^{E}$, decide whether \ed$(G,d)\le k$.}
\medskip

This problem has been extensively studied (e.g. in \cite{B07,BC07}) and its complexity is well understood. In particular, using a reduction from the 3SAT problem,  Saxe \cite{Saxe} shows the following complexity result.

\begin{theorem}\cite{Saxe} \label{theoSaxe}
For any fixed $k\ge 1$, deciding whether \ed$(G,d) \le k$ is $\NP$-hard, already when restricted to weights $d\in \{1,2\}^E$.
\end{theorem}


We now recall a well known connection between Euclidean and Gram realizations.  Given a graph $G=(V,E)$ and its suspension graph $\nabla G$, consider the one-to-one  map $\phi:  \oR^{V \cup E} \mapsto \oR^{E(\nabla G)}$,  which maps  $x\in \oR^{V \cup E}$ to
 $\varphi(x )=d \in \oR^{E(\nabla G)}$ defined  by
\begin{equation}\label{eqphixd}
d_{0i}= x_{ii}  \ (i\in [n]),\ \ \ d_{ij}=x_{ii}+x_{jj}-2x_{ij} \ (\{i,j\}\in E).
\end{equation}
  Then the vectors $u_1,\ldots,u_n\in\oR^k$ form a Gram representation of $x$ if and only if the vectors  $u_0=0,u_1,\ldots,u_n$ form a Euclidean representation of $d=\varphi(x)$ in $\oR^k$. This implies the following:
  
 \begin{lemma}\label{covariance}
Let $G=(V,E)$ be a graph and $x\in \EE(G)$. Then,  $$\gd(G,x)={\rm ed}(\nabla G,\varphi(x)).$$
 \end{lemma}
 As we will see in the next section, this connection will enable us to recover the above result of Saxe for the case $k\ge 3$ from results about the Gram dimension (cf. Corollary \ref{corSaxe}).
 
\ignore{
 Notice that the image $d=\phi(\bfO)$ of the zero  vector  under the  map $\phi$ satisfies   $d_{0i}=1$ for $ i \in [n]$ and $d_{ij}=2$ for $ \{i,j\} \in E$. Combining Lemma~\ref{covariance} with our complexity results concerning $\gd(G,x)$ from  Section~\ref{secEk} we recover    Saxe's complexity results for $k\ge 3$. 
 
 \begin{corollary} 
 For  fixed $k\ge 3$,  it is an $\NP$--hard problem to decide whether \ed$(G,d)~\le~k$, 
  already for $G=\nabla^{k-2}H$ where $H$ is planar, and for $d$ having all its entries   equal to 2 except $d_{0i}=1$ for one apex node $0$.
  
  For $k=2$,  it is NP-hard to decide whether \ed$(G,d)\le 2$,  already when $G$ is a wheel graph, i.e., the suspension of a circuit. 
 \end{corollary}
 
}

\section{Testing  membership in $\EE_k(G)$}\label{secEk}

 In this section we discuss the complexity of testing membership in the rank constrained elliptope $\EE_k(G)$. Specifically, for  fixed $k\ge 1$  we consider the following problem:

\medskip\noindent
 {\em  Given a graph $G=(V,E)$ and  $x\in \oQ^E$, decide whether $\gd(G,x)\le k$.}

\medskip\noindent
In the language of matrix completions this corresponds to deciding whether a rational $G$-partial matrix has a psd completion of rank at most $k$. 

For $k=1$, $x\in \EE_1(G)$ if and only if $x\in \{\pm 1\}^E$ corresponds to a cut of $G$, and it is an easy exercise that this can be decided in polynomial time.
In this section we show that the problem is $\NP$-hard  for any  $k\ge 2$. It turns out that we have to use different reductions for the cases  $k\ge 3$ and  $k=2$. 

\subsection{The case $k\ge 3$}\label{seck3}
First we consider the problem of testing membership in $\EE_k(G)$ when $k\ge 3$. 
We  show this  is an $\NP$-hard problem, already when  $G=\nabla ^{k-3}H$ is the  suspension   of a planar graph  $H$ and  $x=\bfO$, the all-zero vector.

The key idea is to relate the parameter $\gd(G,\bfO)$ to the  chromatic number  $\chi(G)$ (the minimum number of colors needed to color the nodes of $G$ in such a way that adjacent nodes receive distinct colors). It is easy to check that 
 \begin{equation}\label{eqcolor}
 \gd(G,\bfO)\le \chi(G),
 \end{equation}
 with equality if $\chi(G)\le 2$ (i.e., if $G$ is a bipartite graph).
For $k \ge 3$  the inequality (\ref{eqcolor}) can  be strict. This is the case, e.g., for 
orthogonality graphs of  Kochen-Specker sets (see \cite{ovc10}). 

However, Peeters \cite[Theorem 3.1]{P96} gives a polynomial time reduction of  the problem of deciding 3-colorability of a graph to that of deciding  $\gd(G,\bfO)\le 3$.
Namely, given a graph $G$, he   constructs (in polynomial time) a new graph $G'$  having the property that
\begin{equation}\label{eqG}
\chi(G)\le 3  \Longleftrightarrow \chi(G')\le 3 \Longleftrightarrow \gd(G',\bfO)\le 3.
\end{equation}
The graph $G'$ is obtained from $G$ by adding for each pair of distinct nodes $i,j\in V$   the gadget graph $H_{ij}$  shown in Figure~\ref{gadget}.
\begin{figure}[h] 
\centering \includegraphics[scale=0.5]{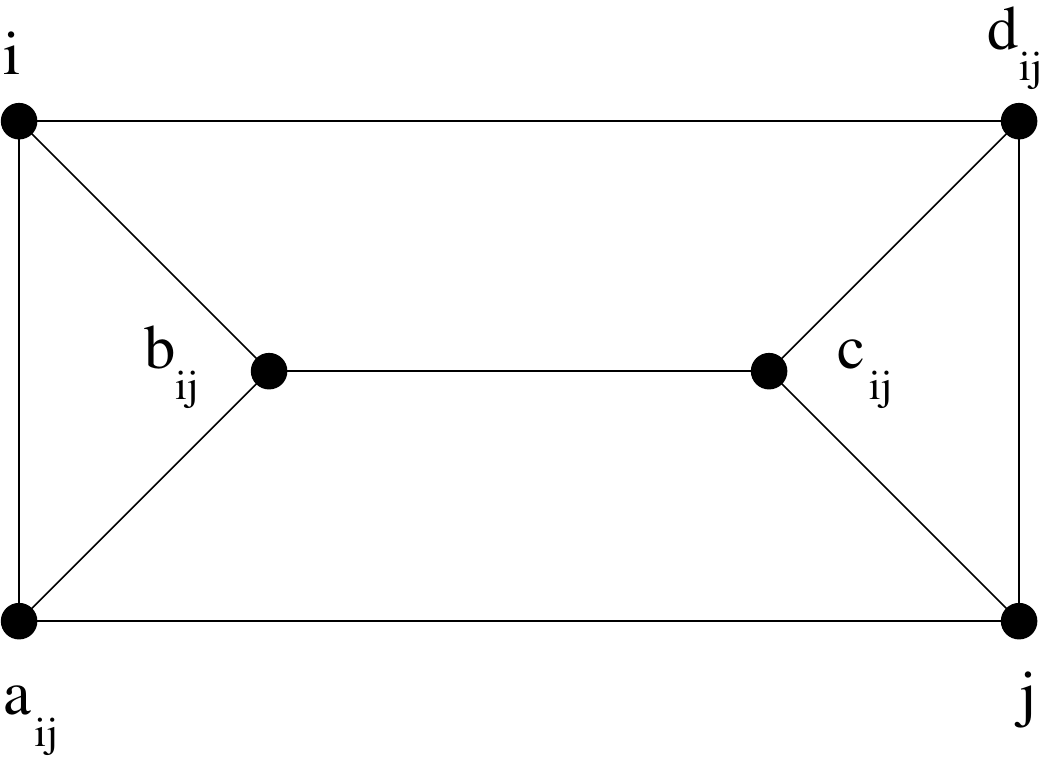}
\caption{The gadget graph $H_{ij}.$} 
\label{gadget} 
\end{figure}
Moreover, using a more involved construction,   Peeters \cite{P95} constructs (in polynomial time) from any  graph  $G$   a new  {\em planar} graph $G'$ satisfying (\ref{eqG}).
As the  problem of deciding whether a given planar graph is 3-colorable is  $\NP$-complete  (see \cite{S73})  we have   the following result.

\begin{theorem}\cite{P95}\label{thm:peeters}
It is $\NP$-hard to decide  whether $\gd(G,\bfO)\le 3$,   already  for the class of planar graphs.
\end{theorem}

This hardness result    can be extended  to any fixed $k \ge 3$ using the suspension operation on graphs. 
The {\em suspension}  graph $\nabla^pG$ is  obtained from $G$ by adding $p$ new nodes that are pairwise adjacent and that are adjacent to all the nodes of $G$.  It is easy observation  that
\begin{equation}\label{eq:1}
\gd(\nabla^p G,\bfO)=\gd(G,\bfO)+p.
\end{equation}
Theorem~\ref{thm:peeters} combined with equation~(\ref{eq:1}) implies:

\begin{theorem} \label{theomaink3}
Fix  $k \ge 3$.  It is $\NP$-hard to decide whether $\gd(G,\bfO)\le k$, already for graphs of the form $G=\nabla^{k-3}H$ where $H$ is a planar graph.
\end{theorem}

 
 As an application we can recover the complexity result of Saxe from Theorem \ref{theoSaxe} for the case $k\ge 3$.
 
 \begin{corollary} \label{corSaxe}
 For  fixed $k\ge 3$,  it is an $\NP$--hard problem to decide whether \ed$(G,d)~\le~k$, 
  already when $G=\nabla^{k-2}H$ with $H$  planar and $d$ is $\{1,2\}$-valued (more precisely, all edges adjacent to a given apex node have weight 1 and all other edges have weight 2).
\end{corollary}

\begin{proof}
This follows directly from Lemma \ref{covariance} combined with Theorem \ref{theomaink3}: By Lemma \ref{covariance},
$\gd(\nabla^{k-3}H,\bfO)=\ed(\nabla^{k-2}H, \varphi(\bfO))$ and observe that  the image $d=\varphi(\bfO)$ of the zero vector under the map $\varphi$ from (\ref{eqphixd}) satisfies: $d_{0i}=1$ and $d_{ij}=2$ for all nodes $i,j$ of $\nabla^{k-3}H$.
\qed\end{proof}

\subsection{The case $k=2$}\label{seck2}
In this section we show $\NP$-hardness of testing membership in $\EE_2(G)$. 
Our strategy to show this result is as follows: Given a graph $G=(V,E)$ with edge weights $d\in \oR^E_+$, define the new edge weights $x=\cos (d)\in \oR^E$.
We show a close relationship between the two problems of testing whether $\ed(G,d)\le 1$, and  whether $\gd(G,x)\le 2$ (or, equivalently, $x\in \EE_2(G)$).  More precisely, we show that each of these two properties can be characterized in terms of the existence of a $\pm 1$-signing of the edges of $G$ satisfying a suitable `flow conservation' type property; moreover, both are equivalent when the edge weights $d$ are small enough.

As a motivation, let us consider first the case when $G=C_n$ is a circuit of length $n$.
Say, weight $d_i$ (resp., $x_i=\cos d_i$) is assigned to the edge $(i,i+1)$ for $i\in [n]$ (setting $n+1=1$).
Then the following property holds:
\begin{equation}\label{eqed1}
\ed(C_n,d)\le 1 \Longleftrightarrow \exists \epsilon \in \{\pm 1\}^n \ \text{ such that }  \epsilon^\sfT d=0.
\end{equation}
This is the key fact used by   Saxe \cite{Saxe} for showing  $\NP$-hardness of the problem of testing $\ed(C_n,d)\le 1$ 
by reducing it from the Partition problem for   $d=(d_1,\cdots,d_n)\in \oZ^n_+$. 
Lemma \ref{lemE2Cn} shows the analogous result for the Gram dimension:
\begin{equation}\label{eqgd2}
\gd(C_n,\cos d)\le 2 \Longleftrightarrow \exists \epsilon\in \{\pm 1\}^n \ \text{ such that } \epsilon^\sfT d\in 2\pi\oZ.
\end{equation}
We now observe that these two characterizations extend for an arbitrary 
graph $G$. 
To formulate the result we need to fix an (arbitrary) orientation $\tG$ of $G$.
Let  $P=(u_0,u_1,\cdots, u_{k-1},u_k)$ be a walk in $G$, i.e., 
 $\{u_i,u_{i+1}\}\in E$ for all $i\le k-1$. 
For $\epsilon\in \{\pm 1\}^E$, we define the following  weighted sum along the edges of $P$:
\begin{equation}\label{eqphi}
\phi_{d,\epsilon}(P)=\sum_{i=0}^{k-1} d_{u_i,u_{i+1}} \epsilon_{u_{i}u_{i+1}} \eta_i,
\end{equation}
setting $\eta_i=1$ if the edge $\{u_i,u_{i+1}\}$ is oriented in $\tG$  from $u_i$ to $u_{i+1}$ and $\eta_i=-1$ otherwise.

\begin{lemma}\label{lemed1}
Consider a graph $G=(V,E)$ with edge weights $d\in\oR^E_+$ and fix an orientation $\tG$ of $G$.
The following assertions are equivalent.
\begin{description}
\item[(i)]
$\ed(G,d)\le 1.$
\item[(ii)]
There exists an edge-signing $\epsilon \in \{\pm 1\}^E$ for which the function $\phi_{d,\epsilon}$ from (\ref{eqphi}) satisfies:
$\phi_{d,\epsilon}(C)=0$ for all circuits $C$ of $G$.
\end{description}
\end{lemma}

\begin{proof}
Assume  that (i) holds. Let $f:V \rightarrow \oR$ satisfying $|f(u)-f(v)|=d_{uv}$ for all  $\{u,v\}\in E$.
If the edge $\{u,v\}$ is oriented from $u$ to $v$ in $\tG$, let $\epsilon_{uv}\in \{\pm 1\}$ such that 
$f(v)-f(u)= d_{uv}\epsilon_{uv}.$
This defines an edge-signing $\epsilon \in \{\pm 1\}^E$; we claim that (ii) holds for this edge-signing. For this, pick a circuit  $C=(u_0,u_1,\cdots,u_k=u_0)$ in $G$. By construction of the edge-signing,  the term $\epsilon_{u_iu_{i+1}} d_{u_iu_{i+1}}\eta_i$ is equal to $f(u_{i+1})-f(u_i)$ for all $i$.
This implies that $\phi_{d,\epsilon}(C)= \sum_{i=0}^{k-1} f(u_{i+1})-f(u_i)=0$ and thus (ii) holds.
Conversely, assume (ii) holds. We may assume  that $G$ is connected (else apply the following to each connected component).
Fix an arbitrary node $u_0\in V$. We define the function $f:V\rightarrow \oR$ by setting $f(u_0)=0$ and, for $u\in V\setminus \{u_0\}$, $f(u)=\phi_{d,\epsilon}(P)$ where $P$ is a walk from $u_0$ to $u$. It is easy to verify that since (ii) holds this definition does not depend on the choice of $P$.
We claim that $f$ is a Euclidean embedding of $(G,d)$ into $\oR$.
For this, pick an edge $\{u,v\}\in E$; say, it is oriented from $u$ to $v$ in $\tG$. 
Pick a walk $P$ from $u_0$ to $u$, so that $Q=(P,v)$ is a walk from $u_0$ to $v$. Then,  $f(u)=\phi_{d,\epsilon}(P)$, $f(v)=\phi_{d,\epsilon}(Q)=
\phi_{d,\epsilon}(P)+ d_{uv}\epsilon_{uv}=f(u)+d_{uv}\epsilon_{uv}$, which implies that
$|f(v)-f(u)|=d_{uv}$.
\qed
\end{proof}

\begin{lemma}\label{lemgd2}
Consider a graph $G=(V,E)$ with edge weights $d\in\oR^E_+$  and fix an orientation $\tG$ of $G$.
The following assertions are equivalent.
\begin{description}
\item[(i)]
$\gd(G,\cos d)\le 2.$
\item[(ii)]
There exists an edge-signing $\epsilon \in \{\pm 1\}^E$ for which the function $\phi_{d,\epsilon}$ from (\ref{eqphi}) satisfies:
$\phi_{d,\epsilon}(C)\in 2\pi\oZ$ for all circuits $C$ of $G$.
\end{description}
\end{lemma}

\begin{proof}
Assume (i) holds. That is, there exists a labeling of the nodes $u\in V$ by unit vectors $g(u)=(\cos f(u), \sin f(u))$ where $f(u)\in [0,2\pi]$.
For any edge $\{u,v\}\in E$, we have:
$\cos d_{uv}= g(u)^\sfT g(v) = \cos(f(u)-f(v))$.
If $\{u,v\}$ is oriented from $u$ to $v$, define $\epsilon \in \{\pm 1\}$ such that $f(v)-f(u)- \epsilon_{uv}d_{uv} \in 2\pi \oZ$.
This defines an edge-signing $\epsilon\in \{\pm 1\}^E$ which satisfies (ii) (same argument  as in the proof of  Lemma \ref{lemed1}).

Conversely, assume (ii) holds. Analogously to the proof of Lemma \ref{lemed1}, fix a node $u_0\in V$ and consider  the unit vectors  $g(u_0)=(1,0)$ and 
$g(u)= (\cos (\phi_{d,\epsilon}(P_u)),\sin (\phi_{d,\epsilon}(P_u)))$, where $P_u$ is a walk from  $u_0\in V$ to $u\in V\setminus\{u_0\}$; one can verify  that these vectors   form a Gram realization of $(G,\cos d)$.
\qed\end{proof}

\begin{corollary}\label{coredgd}
Consider a graph $G=(V,E)$ with edge weights $d\in\oR^E_+$ satisfying $\sum_{e\in E} d_e<2\pi$. Then, $\ed(G,d)\le 1$ $\Longleftrightarrow $ $\gd(G,\cos d)\le 2$.
\end{corollary}

\begin{proof}
Note that if $C$ is a circuit of $G$, then $\phi_{d,\epsilon}(C)\in 2\pi\oZ$ implies $\phi_{d,\epsilon}(C)=0$,  since 
$|\phi_{d,\epsilon}(C)|\le \sum_{e\in E}d_e <2\pi$.
The result now follows directly by applying 
Lemmas \ref{lemed1} and \ref{lemgd2}.
\qed\end{proof}

We can now  show $\NP$-hardness of  testing membership in the rank constrained elliptope $\EE_2(G)$.
For this we use the result of Theorem \ref{theoSaxe} for the case $k=1$: Given edge weights $d\in \{1,2\}^E$, it is $\NP$-hard to decide whether $\ed(G,d)\le 1$.

\begin{theorem}\label{theogd2}
Given a graph $G=(V,E)$ and rational  edge weights $x\in \oQ^E$, it is $\NP$-hard to decide whether $x\in \EE_2(G)$ or, equivalently, $\gd(G,x)\le 2$.
\end{theorem}

\begin{proof}
Fix  edge weights $d\in \{1,2\}^E$. We  reduce the problem of testing whether $\ed(G,d)\le 1$ to the problem of testing whether $\gd(G,\cos (\alpha d))\le 2$, where $\alpha$ is  chosen in such a way that $\cos \alpha \in \oQ$ and $\alpha <1/(\sum_{e\in E}d_e)$. 
For this, set $D=\sum_{e\in E}d_e$ and define the angle $\alpha >0$ by
$$\cos \alpha ={16 D^2-1\over 16 D^2+1}\in \oQ, \ \ \sin\alpha ={8D\over 16 D^2+1}\in \oQ.$$
Then, $\sin\alpha <1/(2D) \le 0.5<\sin 1$, which implies that $\alpha <2\sin\alpha \le 1/D$ and thus $\alpha <1/D=1/(\sum_{e\in E}d_e)$.

As $d_e\in \{1,2\}$,     $\cos( \alpha d_e )\in\{\cos \alpha, \cos(2\alpha)=2\cos^2\alpha  -1 \}$ is rational valued for all edges $e\in E$.
As $\sum_{e\in E}\alpha d_e<1 <2\pi$,   Corollary \ref{coredgd} shows that 
$\gd(G,\cos(\alpha d))\le 2$ is equivalent to $\ed(G,\alpha d)\le 1$ and thus to $\ed(G,d)\le 1$. This concludes the proof.
\qed\end{proof}

We conclude with a remark about the complexity of  the Gram dimension of weighted circuits.

\begin{remark}
Consider the case when $G=C_n$ is a circuit and the edge weights $d\in \oZ^{C_n}_+$ are integer valued.
Relation (\ref{eqed1}) shows that  $\ed(C_n,d)\le 1$ if and only if the sequence $d=(d_1,\cdots,d_n)$ can be partitioned, thus showing $\NP$-hardness of the problem of testing $\ed(C_n,d)\le 1$.

As in the proof of Theorem \ref{theogd2} let us choose $\alpha $ such that $\cos \alpha,\sin \alpha \in \oQ$ and $\alpha<1/(\sum_{i=1}^nd_i)$; then $\cos(t\alpha)\in \oQ$ for all $t\in\oZ$. The
 analogous relation (\ref{eqgd2}) holds, which shows that $\gd(C_n,\cos (\alpha d))\le 2$ if and only if the sequence $d=(d_1,\cdots,d_n)$ can be partitioned. However,  it is not clear how to use this fact  in order  to show $\NP$-hardness of the problem of testing 
 $\gd(C_n,x)\le 2$.
Indeed, although all $\cos(\alpha d_i)$ are rational valued, the difficulty is  that it is  not clear how to compute $\cos (\alpha d_i)$ in time polynomial in the bit size of $d_i$ (while  it can be shown to be polynomial in $d_i$).

Finally we point out the following link to the construction of Aspnes et al. \cite[\S IV]{Aspnes}.
Consider the edge weights $x=\cos(\alpha d)\in\oR^{C_n}$ for the circuit $C_n$ and $y=\varphi(x)$ for its suspension $\nabla C_n$, which is the wheel graph $W_{n+1}$. Thus $y_{0i}=1$ and $y_{i, i+1}=2-2\cos (\alpha d_i)= 4 \sin^2(\alpha d_i/2)$ for all $i\in [n]$. Taking square roots we find the edge weights used in
 \cite{Aspnes} to claim $\NP$-hardness of realizing  weighted wheels (that have the property of admitting unique (up to congruence) realizations in the plane). 
 As explained above in the proof of Theorem \ref{theogd2}, if we suitably choose $\alpha$ we can make sure that 
 all $\sin (\alpha d_i/2)$ be rational valued, while \cite{Aspnes}  uses real numbers. However, it is not clear  how to control their bit sizes, and thus how to deduce $\NP$-hardness.
\end{remark}

\section{Testing  membership in $\convEk$ }\label{secconvEk}

In the previous section we showed that  testing membership in the rank constrained elliptope $\EE_k(G)$ is an $\NP$-hard problem  for any fixed $k\ge 2$.
A related question is to determine  the complexity of optimizing a linear objective function over $\EE_k(G)$ or, equivalently, over its convex hull $\convEk$. 
This question has been  raised, in particular, by Lov\'asz \cite[p.61]{Lo01} and  more recently in~\cite{BOV11}, and 
we will come back to it in Section \ref{secfinal}.  In turn, this  is related to the  problem of testing membership in the convex hull $\convEk$ which we address in this section. Specifically,  for any fixed $k \ge 1$ we consider the following problem:

 \medskip\noindent
 {\em  Given a graph $G=(V,E)$ and $x\in \oQ^E$, decide whether $x\in \convEk$.}
 \medskip

\ignore{ Let $k\ge 1$ be a fixed integer. In this section we discuss the complexity of testing membership in $\conv \EE_k(G)$, i.e., the complexity of the decision problem: 
\begin{align*}
{\rm Input}:& \  {\rm Graph } \ G=([n],E)\   {\rm and} \  x \in \oQ^E. 
\\{\rm Output}:&\ {\rm Does } \ x\in \conv \EE_k(G)\ ?
\end{align*}
Our main objective in this section is to show that this problem is NP-hard for any fixed $k\ge 2$ already when restricted to $x \in \ext(G)\cap \oQ^E$. 
}
For $k=1$, $\conv \spp \EE_1(G)$ coincides with the cut polytope of $G$, for which the membership problem is   $\NP$-complete~\cite{AD91}.
In this section we will show that this problem is $\NP$-hard for any fixed $k\ge 2$.
The key fact to prove  hardness  is to consider the membership problem in  $\convEk$ for extreme points of the elliptope $\EE(G)$. 

For a convex set $K$ recall that a point $x\in K$ is an {\em extreme point} of $K$ if $x=\lambda y+(1-\lambda)z$ where $0<\lambda <1$ and $y,z\in K$ implies that $x=y=z$.  The set of extreme points of $K$ is denoted by $\ext \spp K$.
Clearly, for   $x \in \ext \spp \EE(G)$, 
\begin{equation}\label{eqext}
 x \in \convEk \Longleftrightarrow x \in \EE_k(G).
\end{equation}

\medskip
Our strategy for showing hardness of membership in $\convEk$ is as follows: Given a graph $G=(V,E)$ and a rational vector $x\in \EE(G)$, we construct (in polynomial time) a new graph $\haG=(\haV, \haE)$ (containing $G$ as a subgraph) and a new rational vector $\hax\in\oQ^{\haE}$ (extending $x$) satisfying the following properties:
\begin{equation}\label{eqprop1}
\hax\in  \ext\spp \EE(\haG),
\end{equation}
\begin{equation}\label{eqprop2}
x\in \EE_k(G)\Longleftrightarrow \hax \in \EE_k(\haG).
\end{equation}
Combining these two conditions with (\ref{eqext}), we deduce:
\begin{equation}\label{eqprop3}
x\in \EE_k(G) \Longleftrightarrow \hax \in \EE_k(\haG) \Longleftrightarrow \hax\in \conv \spp \EE_k(\haG).
\end{equation}

\begin{figure}[h]
\centering \includegraphics[scale=1]{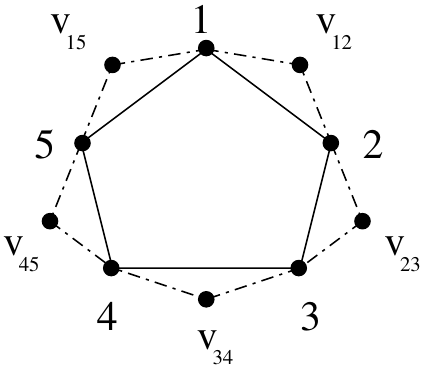}
\caption{The graph $\widehat C_5$.}
\label{graph_triang} 
\end{figure}

 Given $G=(V,E)$, the construction of the new graph $\haG=(\haV,\haE)$ is as follows: For each edge $\{i,j\}$ of $G$, we add a new node $v_{ij}$, adjacent to the two nodes $i$ and $j$. Let $C_{ij}$ denote  the clique on $\{i,j,v_{ij}\}$ and set $\haV=V\cup\{v_{ij}: \{i,j\}\in E\}$.
 Then $\haG$ has node set $\haV$ and its edge set is the union of all the cliques $C_{ij}$ for $\{i,j\}\in E$.  As an illustration Figure \ref{graph_triang} shows the graph $\widehat{C_5}$.

Given $x\in \oQ^E$, the construction of the new vector $\hax\in \oQ^{\haE}$ is as follows: For each edge $\{i,j\}\in E$,
\begin{equation}\label{eqxhat1}
\hax_{ij}=x_{ij},
\end{equation}
\begin{equation}\label{eqxhat2}
\hax_{i,v_{ij}}=4/5, \ \hax_{j,v_{ij}}=3/5 \ \ \ \  \ \   \ \text{ if } x_{ij}=0,
\end{equation}
\begin{equation}\label{eqxhat3}
\hax_{i, v_{ij}}=x_{ij},\ \hax_{j,v_{ij}}=2x_{ij}^2-1 \ \  \text{ if } x_{ij}\ne 0.
\end{equation}

 \medskip
 We will use the following result characterizing the extreme points of the elliptope $\EE_3$.

\begin{theorem}\label{thm:extE3}\cite{GPW90} A matrix $X=(x_{ij})\in \EE_3$ is an extreme point of $\EE_3$ if either $\rank (X)=1$,
   or $\rank (X) =2$ and $|x_{ij}|<1$ for all $i\not=j\in \{1,2,3\}$.
\end{theorem}

We also need the following well known (and easy to check) result permitting to construct points in the elliptope of clique sums of graphs.

\begin{lemma}\label{lemcliquesum}
Given two graphs $G_l=(V_l,E_l)$ ($l=1,2$), where $V_1\cap V_2$ is a clique in both $G_1$, $G_2$,  the graph 
$G=(V_1\cup V_2,E_1\cup E_2)$ is called  their {\em clique sum}.  Given $x_l\in \oR^{E_l}$ ($l=1,2$) such that $(x_1)_{ij}=(x_2)_{ij}$ for $i,j \in V_1\cap V_2$,  let  $x=(x_{ij}) \in \oR^E$ be their common extension, defined as  $x_{ij}=(x_l)_{ij}$ if $i,j\in V_l$. Then, for any integer $k\ge 1$, 
$$x\in \EE_k(G) \Longleftrightarrow  x_1\in \EE_k(G_1) \text{ and } x_2\in \EE_k(G_2).$$
\end{lemma}

We can now show that our construction for $\hax$ satisfies the two properties (\ref{eqprop1}) and (\ref{eqprop2}).

\begin{lemma}\label{lemreduc}
Given a graph $G=(V,E)$ and $x\in \oQ^E$, let $\haG=(\haV,\haE)$ be defined as above and   let $\hax\in \oQ^{\haE}$ be defined by (\ref{eqxhat1})-(\ref{eqxhat3}). 
For fixed $k\ge 2$ we have that  $x\in \EE_k(G)$ if and only if   $\hax\in \EE_k(\haG)$ and $\hax\in \ext \spp \EE(\haG)$.\end{lemma}

\begin{proof} Sufficiency follows  trivially so it remains to prove necessity. Observe that, for each edge $\{i,j\}\in E$,  the restriction $\hax_{C_{ij}}$ of $\hax$ to the clique $C_{ij}$ is an extreme point of $\EE(C_{ij})$.  
Indeed, applying Theorem \ref{thm:extE3}, we find that  the following matrices 
 $$
 \left(\begin{array}{ccc}
 1 & 0 & 3/5\cr 0 & 1 & 4/5 \cr 3/5 & 4/5 & 1 
  \end{array}\right),\ 
    \left(\begin{array}{ccc}
 1 &  x_{ij}&  x_{ij}\\
 x_{ij} & 1& 2x_{ij}^2-1\\
 x_{ij}& 2x_{ij}^2-1& 1
 \end{array}\right) \ \text{ where } x_{ij}\in [-1,1]\setminus \{0\}
 $$
are  extreme points of $\EE_3$ so have  rank at most 2.
By construction,   $\haG$ is obtained as the clique sum of $G$ with the cliques $C_{ij}$. Hence, by Lemma \ref{lemcliquesum},
  we deduce that $\hax\in \EE_k(\haG)$. 
  
  Finally, we show that $\hax$ is an extreme point of $\EE(\haG)$, which follows from the fact that each $\hax_{C_{ij}}$ is an extreme point of $\EE(C_{ij})$, combined with  the fact that the cliques $C_{ij}$ (for $\{i,j\}\in E$) cover the graph $G$. Indeed, assume $\hax=\sum_{i=1}^m\lambda_i\hax_i$ where $\lambda_i> 0$, $\sum_{i=1}^m\lambda_i=1$  and $\hax_i\in \EE(\haG)$. Taking the projection onto the clique $C_{ij}$ and using the fact that $\hax_{C_{ij}} \in \ext\spp \EE(C_{ij})$ we deduce that, for all $k$,  $(\hax_k)_{C_{ij}}=\hax_{C_{ij}}$ for all $\{i,j\}\in E$ and thus $\hax=\hax_k$. \qed
  \end{proof}

Combining Theorems \ref{theomaink3} and \ref{theogd2}  with Lemma \ref{lemreduc} and relation  (\ref{eqprop3}) we deduce the following complexity result.

\begin{theorem} \label{theoextgd}
For any fixed $k\ge 2$, testing membership in $\convEk$ is an $\NP$-hard problem.
 
 \end{theorem}
 
 \section{Concluding remarks}\label{secfinal}
 
 In this note we have shown $\NP$-hardness of the membership problem in the  rank constrained  elliptope $\EE_k(G)$ and in its convex hull $\convEk$,  for any fixed $k\ge 2$.
 As mentioned earlier, it would be interesting to settle the complexity status of linear optimization over $\convEk$.
 The  case $k=1$ is settled: Then  $\conv \spp \EE_1(G)$ is the cut polytope and both the membership problem and the linear optimization problem are $\NP$-complete. 
 For $k\ge 2$, the convex set $\convEk$ is  in general non-polyhedral. Hence the right question to ask is about the complexity of  the {\em weak} optimization problem. It follows from general results about the ellipsoid method (see, e.g., \cite{GLS} for details) that the weak optimization problem and the weak membership problems for $\convEk$  have the same complexity status.
 Although we could prove that the (strong) membership problem in $\convEk$ is $\NP$-hard, we do not know whether this is also the case for the  {\em weak} membership problem.
 
 A second question of interest  is whether the problems belong to $\NP$. Indeed it is not clear how to find {\em succinct}  certificates for membership in $\EE(G)$ or in $\EE_k(G)$. For one thing, even if the given partial matrix $x$ is rational valued and is completable to a psd matrix, it is not known whether it admits a {\em rational}  completion.  (A positive result has been shown in \cite{Lau00} in the case of chordal graphs, and for graphs with minimum fill-in 1).
 In a more general setting, it is not known whether the problem of testing feasibility of a semidefinite program belongs to $\NP$.
On the positive side it is known that this problem belongs to $\NP$ if and only if it belongs to co-$\NP$ \cite{Ra97} and that it can be solved in polynomial time when fixing the dimension or the number of constraints \cite{PK97}.

\medskip
{\bf Acknowledgements.} We thank A. Schrijver for useful discussions and a referee for drawing our attention to the paper by Aspnes et al. \cite{Aspnes}.

\end{document}